\documentclass[12pt,a4paper,dvipsnames]{article}

\usepackage{hyperref}
\usepackage{pgf}
\usepackage[utf8]{inputenc}

\usepackage{mathtools}
\usepackage{amsfonts}
\usepackage{amssymb}
\usepackage{amsthm}
\usepackage{dsfont}
\usepackage[linesnumbered]{algorithm2e}
\usepackage{setspace}
\usepackage{extarrows}

\usepackage{graphicx}
\usepackage{multicol}
\PassOptionsToPackage{dvipsnames}{xcolor}
\usepackage{tikz,xcolor}

\usetikzlibrary{shapes,arrows}
\usetikzlibrary{hobby}

\usepackage[english]{babel}		
\usepackage[T1]{fontenc}			
\usepackage{color}         			
\usepackage{listings}					

\renewcommand{\subsectionmark}[1]{}

\newcommand{\dotcup}{\ensuremath{\mathaccent\cdot\cup}}

\usepackage{geometry}
\newcommand*\dif{\mathop{}\!\mathrm{d}}

\newtheorem{theorem}{Theorem}[section]
\newtheorem{lemma}[theorem]{Lemma}

\newtheorem{proposition}[theorem]{Proposition}

\newtheorem{corollary}[theorem]{Corollary}
\newtheorem{remark}[theorem]{Remark}
\renewcommand{\P}{\mathbb{P}}

\newcommand{\N}{\mathbb{N}}
\newcommand{\R}{\mathbb{R}}
\newcommand{\E}{\mathbb{E}}

\newcommand\abs[1]{\left| #1\right|}

\title{Limit theorems for the tagged particle in exclusion processes on regular trees}
\author{
Dayue Chen$^\ast$, Peng Chen$^\ast$, 
Nina Gantert$^{\ast\ast}$ and Dominik Schmid$^{\ast\ast}$}  
\date{\today}

\begin{document}


\maketitle

\abstract{We consider exclusion processes on a rooted $d$-regular tree. We start from a Bernoulli product measure conditioned on having a particle at the root, which we call the tagged particle. For $d\geq 3$, we show that the tagged particle has positive linear speed and satisfies a central limit theorem. We give an explicit formula for the speed. As a key step in the proof, we first show that the exclusion process ``seen from the tagged particle'' has an ergodic invariant measure.}

\phantom{.} \hspace{0.2cm} \textbf{Keywords:} Exclusion process, regular tree, tagged particle, ergodicity \\
\phantom{.} \hspace{0.5cm} \textbf{AMS 2000 subject classification:} 60K35

\let\thefootnote\relax
\footnotetext{ $^\ast$ \textit{Peking University, China. E-Mail}: \nolinkurl{dayue@pku.edu.cn}, \nolinkurl{chenpeng@cufe.edu.cn} \\\phantom{.} \hspace{0.3cm}
$^{\ast\ast}$ \textit{TU München, Germany. E-Mail}: \nolinkurl{nina.gantert@tum.de}, \nolinkurl{dominik.schmid@tum.de}}

\section{Introduction}

The simple exclusion process is one of the most studied examples of an interacting particle system.
It can model a variety of systems as for instance cars in a traffic jam or molecules in a low-density gas.
The exclusion process is given by a locally finite graph and a set of indistinguishable particles, which we initially place on distinct sites of the graph. Each particle independently performs a simple random walk.
However, when a particle would move to an occupied site, the move is suppressed.
In the following, we assume that the graph has a root on which we initially place a particle.
We then follow the evolution of this tagged particle. Our goal is to study its distance from the root.
We will focus on the case where the underlying graph is a rooted $d$-regular tree for $d\geq 2$ and establish a law of large numbers with respect to the shortest path distance from the root. Moreover, for $d \geq 3$, we show that a central limit theorem holds.

\subsection{The model}

For $d\in \N$ with $d\geq 2$, let $T^d=(V,E)$ denote the $\boldsymbol{d}$\textbf{-regular tree} with a distinguished site $o$ which we call the \textbf{root}. Consider transition rates $\tilde{p}(x,y)\in \R_0^+$ for all $x,y \in V$, and assume they are symmetric, translation invariant, irreducible and of finite range. In particular, for all sites $x,y \in V$, the transition rates have to satisfy $\tilde{p}(x,y)=p(|x-y|)$ for some function $p:\N_0 \rightarrow \R^+_0$ of finite support, where $|x-y|$ denotes the shortest path distance of $x$ and $y$ in $T^d$.
We define the \textbf{exclusion process on} $\boldsymbol{T^d}$ to be the Feller process $(\eta_t)_{t \geq 0}$ with state space $\{0,1\}^V$ generated by the closure of
\begin{align}\label{def:generatorEta}
\mathcal{L}f(\eta) &= \sum_{x,y \in V} \tilde{p}(x,y)\  \eta(x)(1-\eta(y))\left[ f(\eta^{x,y})-f(\eta) \right] \, .
\end{align}
Here, $\eta^{x,y}$ denotes the configuration where we exchange the values at positions $x$ and $y$ in $\eta$. In the particular case where $\tilde{p}(x,y)= d^{-1}\mathds{1}_{\{ x,y \} \in E }$ holds for all $x,y \in V$, we call $(\eta_t)_{t \geq 0}$ the \textbf{simple exclusion process on} $\boldsymbol{T^d}$.
Since by our assumptions
\begin{equation*}
\sup_{y \in V} \sum_{x \in V} \tilde{p}(x,y) < \infty
\end{equation*}
holds, \eqref{def:generatorEta} indeed gives rise to a Feller process, see \cite[Theorem 3.9]{liggett1985interacting}.  
For a given configuration $\eta$, we say that a site $x$ is \textbf{occupied} if $\eta(x)=1$ and \textbf{vacant} otherwise. \\
For each $\rho \in [0, 1]$, let $\nu_{\rho}$ denote the Bernoulli-$\rho$-product measure on $\{0, 1\}^V$ with marginal distributions
\begin{equation*}
\nu_{\rho}\left( \eta \colon \eta(x)=1 \right) = \rho
\end{equation*} for all $x \in V$. Under the above assumptions, the collection $\left\lbrace \nu_{\rho} \colon \rho \in [0,1] \right\rbrace$ is exactly the set of extremal invariant measures of $(\eta_t)_{t \geq 0}$, see \cite[III, Theorem 1.10]{liggett1999stochastic}. We define 
\begin{equation*}
\nu_{\rho}^{\ast}( \ . \ ) := \nu_{\rho}\left( \ . \ \mid   \eta(o)=1 \right) 
\end{equation*} to be the Bernoulli-$\rho$-product measure where we condition on having a particle at the root. We call $\nu_{\rho}^{\ast}$ the \textbf{Palm measure} with parameter $\rho$. The particle initially placed on the root is called the \textbf{tagged particle}. We follow the evolution of the tagged particle over time and denote by $(X_t)_{t\geq 0}$ its position. 

\subsection{Main results}\label{sec:main}

Our main result is to establish a law of large numbers for the position of the tagged particle $(X_t)_{t\geq 0}$. In the following, we write $|z|:=|z-o|$ for all $z \in V$. 
\begin{theorem}\label{thm:LLN} For $d\geq 2$, let $(\eta_t)_{t\geq 0}$ on $T^d$ have initial distribution $\nu_{\rho}^{\ast}$ for some $\rho \in [0,1]$. Then the position of the tagged particle $(X_t)_{t\geq 0}$ satisfies a law of large numbers:
\begin{equation*}
\lim_{t \rightarrow \infty} \frac{|X_t|}{t} = (1-\rho)(d-2)\sum_{i \in \N_0}i p(i) =: v
\end{equation*} 
$\P_{\nu_{\rho}^{\ast}}$-almost surely. In particular, we have a speed of $(1-\rho)\frac{d-2}{d}$ in the case of the simple exclusion process on $T^d$.
\end{theorem}

\begin{remark}
The result is not surprising, the same result was proved in \cite{saada1987} for an exclusion process with drift on $\mathbb{Z}^d$.
If $\rho=0$, we obtain the speed of random walk on $T^d$ with transition rates $\tilde{p}(.,.)$, if $\rho =1$, $v=0$ holds and in between the speed is linear in $1-\rho$.
\end{remark}

 For $d\geq 3$, we show that the tagged particle has a diffusive behavior.
\begin{theorem}\label{thm:CLT} For $d\geq 3$ and $\rho \in [0, 1)$, the tagged particle $(X_t)_{t\geq 0}$ on $T^d$ satisfies
\begin{equation*}
\frac{|X_t|-t v }{\sqrt{t}} \ \overset{\text{\normalfont d}}{\longrightarrow} \ \mathcal{N}(0,\sigma^2)
\end{equation*} for some $\sigma=\sigma(d,\rho,p(.))\in (0,\infty)$ and $v$ from Theorem \ref{thm:LLN}.
\end{theorem}  
\begin{remark}
Note that for $d=2$, the $d$-regular tree is $\mathbb{Z}$. In this case, the position of the tagged particle shows a subdiffusive behavior, see \cite{arratia1983}. 
\end{remark}

\subsection{Related work}

The exclusion process can be studied from a variety of different perspectives. To give limit laws for the position of a tagged particle is a classical problem, which was already mentioned in Spitzer’s work \cite{spitzer1970}. In the case where the underlying graph is $\mathbb{Z}^d$, many results were proved.
For translation invariant transition probabilities (which are not concentrated on the nearest neighbors in the one-dimensional case), Saada established a law of large numbers in \cite{saada1987}. 
In Section \ref{sec:environment}, we will closely follow her approach in order to show that the Palm measure is ergodic for the environment process. 
Kipnis and Varadhan established a central limit theorem for the position of the tagged particle with symmetric, translation invariant transition probabilities (which are not concentrated on the nearest neighbors in the one-dimensional case) in their famous paper \cite{kipnis1986central}. The case of nearest neighbor transition probabilities in one dimension was treated before by Arratia and Kipnis \cite{arratia1983,kipnis1986}. In more general non-symmetric translation invariant cases, the position of the tagged particle is diffusive as well, see \cite{sethuraman2006,sethuraman2000,varadhan1995}. For a general introduction to limit theorems for tagged particles, we refer to \cite[III]{liggett1999stochastic} and \cite{komorowski2012}.

\subsection{Outline of the paper}

This paper is organized as follows. 
In Section \ref{sec:environment}, we introduce the environment process, which can be interpreted as the exclusion process ``seen from the tagged particle''. As a key step for the proof of Theorem \ref{thm:LLN}, we show that the Palm measures $\nu_{\rho}^{\ast}$ are ergodic for the environment process, following the approach in \cite{saada1987}. This is a result of independent interest. 
In Section \ref{sec:speed}, we study the tagged particle process in more detail. We show that the tagged particle is transient using a martingale decomposition which can be found in Section III.4 of \cite{liggett1999stochastic}. We then deduce Theorem \ref{thm:LLN} following the ideas of Lyons et al. in \cite{lyons1995}. Section \ref{sec:diffusivity} is dedicated to the proof of Theorem \ref{thm:CLT} using the results of Kipnis and Varadhan as well as Sethuraman et al., see \cite{kipnis1986central,sethuraman2000}.

\section{The environment process}\label{sec:environment}

In order to define the environment process, we first introduce some notation. The $d$-regular tree $T^d=(V,E)$ with root $o$ has a natural interpretation in terms of Cayley graphs. For $I=\{ 1, \dots, d\}$, let 
\begin{equation*}
\mathcal{G} := \langle a_i, i \in I | a_i^2 = e \text{ for all } i \in I \rangle
\end{equation*} denote the free group over all $i \in I$ for the two-element groups $\{e, a_i\}$ with the relation $a_i^2 = e$ and neutral element $e$. The tree $T^d$ can be now be identified with the Cayley graph of $\mathcal{G}$ with respect to the generator $S=\{ a_1, \dots, a_d\}$. Note that the vertex set $V$ is isomorphic to $\mathcal{G}$ with $e \cong o$ and two corresponding elements $b,c \in \mathcal{G}$ are neighbored if and only if $ba = c$ holds for some $a \in S$. The group structure of $T^d$ allows us to extend this relation and define
\begin{equation}\label{eq:groupAddition}
b+c := bc \ \ \text{  as well as  } \ \ b-c := bc^{-1}
\end{equation} for $b,c \in \mathcal{G}$. In the same way, we write $x+y=z$ and  $x-y=z$  for $x,y,z \in V$ if the corresponding elements in $\mathcal{G}$ satisfy \eqref{eq:groupAddition}. Let the maps $\tau_x$ on configurations $\eta \in \{ 0,1\}^{V}$ be given as
\begin{equation*}
\tau_{x}\eta(y) := \eta(x+y)
\end{equation*} 
for all $x,y\in V$. Equipped with these notations, we define the \textbf{environment process} $(\zeta_t)_{t \geq 0}$ as
\begin{equation}\label{def:environment}
\zeta_t(x) := \tau_{X_t}\eta_t(x)
\end{equation} for all $t \geq 0$ and $x \in V$. 
Then $(\zeta_t)_{t \geq 0}$ is again a Feller process on the state space $\{ \zeta \in \{ 0,1\}^V \colon \zeta(o)=1 \}$  generated by the closure of
\begin{align}\label{def:generatorEnvironment}
Lf(\zeta) &= \sum_{x,y \neq o}  p(|x-y|)\zeta(x)(1-\zeta(y))\left[ f(\zeta^{x,y})- f(\zeta)\right] \nonumber \\
&+ \sum_{x \in V}  p(|x|) (1-\zeta(x))\left[ f(\tau_{x}\zeta)- f(\zeta)\right] \ .
\end{align} 
Note that each transition in $(\eta_t)_{t \geq 0}$ involving the root is a transition in $(\zeta_t)_{t \geq 0}$ followed by a translation.
In the following, our goal is to investigate the set of invariant measures of $(\zeta_t)_{t \geq 0}$.
\begin{proposition}\label{pro:invarianceEnvironment}
The measure $\nu^{\ast}_{\rho}$ is invariant for $(\zeta_t)_{t\geq 0}$ for all $\rho \in [0,1]$.
\end{proposition} 
\begin{proof}
See \cite[III, Proposition 4.3]{liggett1999stochastic} for $\mathbb{Z}^d$, which is one-to-one for $T^d$.
\end{proof} 
In order to calculate the speed of the tagged particle, we now show that $(\zeta_t)_{t \geq 0}$ started from $\nu_{\rho}^{\ast}$ 
is a stationary and ergodic process.

\subsection{Ergodicity of the Palm measure}\label{sec:ergodicity}

In order to derive ergodicity with respect to $\nu_\rho^{\ast}$ and $\rho \in (0,1)$, we closely follow the arguments of Saada in \cite{saada1987}. For simplicity of notation, we only consider the case of the simple exclusion process since for general $p(.)$ the same arguments apply.
\begin{proposition}\label{pro:ergodicity}
For $d\geq 3$, the measure $\nu^{\ast}_{\rho}$ is ergodic for $(\zeta_t)_{t\geq 0}$ and all $\rho \in (0,1)$.
\end{proposition}
In order to show Proposition \ref{pro:ergodicity}, suppose that $\nu_{\rho}^{\ast}$ is not ergodic for $(\zeta_t)_{t \geq 0}$. Then we can find a set $A \subseteq \left\lbrace \zeta \in \{ 0,1\}^{V} \colon \zeta(o) = 1\right\rbrace$ such that
\begin{equation}\label{def:setA}
0 < \nu_{\rho}^{\ast}(A) < 1
\end{equation} and $A$ is invariant, i.e. 
\begin{equation*} 
\P^{\zeta}\left( \zeta_t \in A \right) = 1
\end{equation*}
holds for almost all $\zeta \in A$, hence $A$ is a non-trivial invariant set for $(\zeta_t)_{t \geq 0}$. Define $B := \left\lbrace \zeta \in \{ 0,1\}^{V} \colon \zeta(o) = 1 \right\rbrace \setminus A$ and note that $B$ is a non-trivial, invariant set for $(\zeta_t)_{t \geq 0}$ as well. Recall that $\nu_{\rho}$ is extremal invariant for the simple exclusion process $\left(\eta_t \right)_{t \geq 0}$ and hence $\left(\eta_t \right)_{t \geq 0}$ started from $\nu_{\rho}$ is ergodic, see \cite[Theorem B52]{liggett1999stochastic}. We want to use this observation to establish a contradiction. 
Let the sets $\tilde{A}$ and $\tilde{B}$ be given as
\begin{equation*}
\tilde{A} := \bigcup_{x \in V,\ \zeta \in A} \tau_x\zeta \ \ \text{ and }\ \ \tilde{B} := \bigcup_{x \in V,\ \zeta \in B} \tau_x\zeta \ .
\end{equation*}
Then, $\tilde{A}$ and $\tilde{B}$ are invariant for $(\eta_t)_{t \geq 0}$. Since $A \subseteq \tilde{A}$ and $B \subseteq \tilde{B}$, we obtain from \eqref{def:setA} that
\begin{equation}\label{eq:setsABtilde}
\nu_{\rho}(\tilde{A})=\nu_{\rho}(\tilde{B})=1 \ .
\end{equation}
In particular, the sets $\tilde{A}$ and $\tilde{B}$ are not disjoint. From this, we want to deduce that $A$ and $B$ are not disjoint, contradicting the definition of $B$. To do so, we need the following lemma.

\begin{lemma}\label{lem:ergodicity} For almost every $\eta$ distributed according to $\nu_{\rho}$, there exist integers $n,m,l$  and sites
\begin{equation*}
w,x,y,z; \ \  x_1,x_2, \dots ,x_n; \ \ y_1,y_2, \dots ,y_m; \ \  z_1,z_2, \dots ,z_l
\end{equation*} with the following properties:
\begin{itemize}
\item[(i)]\label{item1} $\tau_x\eta \in A$, $\tau_w\eta \in B$
\item[(ii)]\label{item2} $\eta(y)=\eta(z)=\eta(x_1)=\ldots=\eta(x_n)=0$
\item[(iii)]\label{item3} $x,y,z$ are located in pairwise different branches with respect to $w$ in $T^d$
\item[(iv)]\label{item4} $w$ is connected to $x$ via the path $x_1 \sim x_2 \sim \dots \sim x_n$, connected to $y$ via the path $y_1 \sim y_2 \sim \dots \sim y_m$ and connected to $z$ via the path $z_1 \sim z_2 \sim \dots \sim z_l$.
\end{itemize}
\end{lemma}

\begin{proof} By \eqref{eq:setsABtilde}, there almost surely exist sites $x,w \in V$ such that $\tau_x\eta \in A$ and $\tau_w\eta \in B$ holds. Let $x_1,x_2, \dots ,x_n$ denote the shortest path connecting $x$ and $w$, which may be empty for $x \sim w$. Without loss of generality, we assume that $\eta(x_1)=\ldots=\eta(x_n)=0$ holds. More precisely, note that $A$ and $B$ form a partition of $\{ \zeta \in \{0,1\}^V \colon \zeta(o)=1\}$ and so we have that $\tau_w \eta \in A \dotcup B$ holds for all occupied sites $w\in V$. Hence, among the occupied vertices along the path from $x$ to $w$, there exist two sites $\tilde{x},\tilde{y}$ with $\tau_{\tilde{x}}\eta \in A$, $\tau_{\tilde{y}}\eta \in B$ with only vacant sites in between of them. Take $\tilde{x},\tilde{y}$ as new choices for $x$ and $y$. \\
In order to show that properties \hyperref[item3]{(iii)} and \hyperref[item3]{(iv)} hold, let $C(x,w)$ and $D(x,w)$ denote the vertices of two arbitrary branches of $w$ different from the one containing $x$. Since $C(x,w)$ and $D(x,w)$ contain infinitely many sites, for $\nu_{\rho}^{\ast}$-almost every $\eta$ there are infinitely many $y$ in $C(x,w)$ and $z$ in $D(x,w)$ such that $\eta(y) = \eta(z) = 0$ holds. Choose two of these sites as $y$ and $z$ arbitrarily and define $y_1,y_2, \dots ,y_m$ and $z_1,z_2, \dots ,z_l$ to be the shortest paths connecting them to $w$, respectively.
\end{proof}
\begin{proof}[Proof of Proposition \ref{pro:ergodicity}]
Take an $\eta$ satisfying the properties in Lemma \ref{lem:ergodicity} for sites
\begin{equation*}
N:= \left\lbrace w,x,y,z, x_1, x_2, \dots, x_n, y_1,y_2, \dots, y_m, z_1,z_2, \dots, z_l \right\rbrace .
\end{equation*} Fix an arbitrary time $t_0 > 0$. Let $\tilde{\eta}$ denote the configuration that agrees with $\eta$ on $N$ while on the complement of $N$, $\tilde{\eta}$ has the distribution of a simple exclusion process $(\eta_{t})_{t \geq 0}$ at time $t_0$ which is started from $\eta$ and where all moves involving the sites $N$ are suppressed.  In the following, we consider two ways of transforming $\eta$ into $\eta^{x,y}$. Since the transformations use only transitions in $N$, they also provide two ways of transforming $\tilde{\eta}$ into $\tilde{\eta}^{x,y}$ for any fixed $t_0>0$.
\begin{itemize}
\item[(a)]\label{def:way1} First, move the particle from $w$ to $z$ along $z_1,z_2, \dots ,z_l$, i.e. for $\{ i_j, 1\leq j \leq J\}$ being successive values of $i$ such that $\eta(z_{i_j})=1$, move the particle from $z_{i_{J}}$ to $z$, then from $z_{i_{J}}$ to $z_{i_{J-1}}$ and so on. Secondly, move the particle from $x$ to $y$ along $x_1, x_2, \dots, x_n$ and $y_1,y_2, \dots, y_m$ in the same way. Finally, move the particle from $z$ back to $w$ along $z_1,z_2, \dots ,z_l$.
\item[(b)] \label{def:way2} Move the particle from $w$ to $y$ along $y_1,y_2, \dots ,y_m$, then the particle from $x$ to $w$ along $x_1, x_2, \dots, x_n$.
\end{itemize} A visualization of the transformations in \hyperref[def:way1]{(a)} and \hyperref[def:way2]{(b)} is given in Figure \ref{figureLemma}. Note that in \hyperref[def:way1]{(a)}, the particle originally at $w$ moves back to $w$. Since $\tau_w \eta \in B$ and $B$ is invariant for the process $(\zeta_t)_{t \geq 0}$, we conclude that $\tau_w \tilde{\eta}^{x,y} \in B$ holds almost surely. In transformation \hyperref[def:way2]{(b)}, the particle originally at $x$ moved to $w$. Since $\tau_x \eta \in A$ and $A$ is invariant for $(\zeta_t)_{t \geq 0}$, we conclude that $\tau_w \tilde{\eta}^{x,y} \in A$ holds almost surely. Using the graphical representation, observe that $\eta_{t_0}$ agrees with $\tilde{\eta}^{x,y}$  with positive probability. Hence, we obtain a contradiction to $A$ and $B$ being disjoint.
\end{proof}


\section{Speed of the tagged particle}\label{sec:speed}

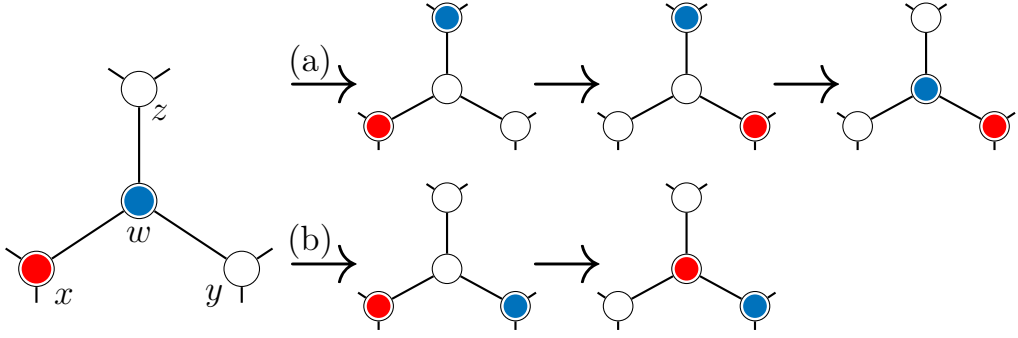
\begin{figure}

\begin{center}
\begin{tikzpicture}[scale=0.9]

	\node[shape=circle,scale=1.2,draw] (A2) at (0,0){} ;
 	\node[shape=circle,scale=1.2,draw] (B2) at (1.5,1){} ;
	\node[shape=circle,scale=1.2,draw] (C2) at (3,0) {};
	\node[shape=circle,scale=1.2,draw] (D2) at (1.5,2.65){} ;
	
\node[scale=1.1]  at (1.5,0.5){$w$};
\node[scale=1.1]  at (0.4,-0.4){$x$};
\node[scale=1.1]  at (1.8,2.3){$z$};
\node[scale=1.1]  at (2.6,-0.4){$y$};

	\draw[thick] (A2) to (B2);		
	\draw[thick] (B2) to (C2);		
	\draw[thick] (B2) to (D2);		

	\draw[thick] (A2) to (0,-0.5);	
	\draw[thick] (C2) to (3,-0.5);	
	\draw[thick] (A2) to (-0.45,0.3);	
	\draw[thick] (C2) to (3.45,0.3);		
	\draw[thick] (D2) to (1.95,2.95);	
	\draw[thick] (D2) to (1.05,2.95);	

 \node[shape=circle,fill=red] (k1) at (A2){};
 \node[shape=circle,fill=RoyalBlue] (k2) at (B2){}; 
 
 	\node[shape=circle,scale=1,draw] (A3) at (5,2.1){} ;
 	\node[shape=circle,scale=1,draw] (D3) at (6,3.7){} ;
	\node[shape=circle,scale=1,draw] (C3) at (7,2.1) {};
	\node[shape=circle,scale=1,draw] (B3) at (6,2.65){} ;

	\draw[thick] (A3) to (B3);		
	\draw[thick] (B3) to (C3);		
	\draw[thick] (B3) to (D3);		
	\draw[thick] (A3) to (5,1.75);	
	\draw[thick] (C3) to (7,1.75);	
	\draw[thick] (A3) to (4.7,2.3);	
	\draw[thick] (C3) to (7.3,2.3);		
	\draw[thick] (D3) to (6.3,3.9);	
	\draw[thick] (D3) to (5.7,3.9);	

 \node[shape=circle,scale=0.8,fill=red] (k1) at (A3){};
 \node[shape=circle,scale=0.8,fill=RoyalBlue] (k1) at (D3){};
 
  	\node[shape=circle,scale=1,draw] (A3) at (8.5,2.1){} ;
 	\node[shape=circle,scale=1,draw] (D3) at (9.5,3.7){} ;
	\node[shape=circle,scale=1,draw] (C3) at (10.5,2.1) {};
	\node[shape=circle,scale=1,draw] (B3) at (9.5,2.65){} ;

	\draw[thick] (A3) to (B3);		
	\draw[thick] (B3) to (C3);		
	\draw[thick] (B3) to (D3);		
	
	\draw[thick] (A3) to (8.5,1.75);	
	\draw[thick] (C3) to (10.5,1.75);	
	\draw[thick] (A3) to (8.2,2.3);	
	\draw[thick] (C3) to (10.8,2.3);		
	\draw[thick] (D3) to (9.8,3.9);	
	\draw[thick] (D3) to (9.2,3.9);	

 \node[shape=circle,scale=0.8,fill=red] (k1) at (C3){};
 \node[shape=circle,scale=0.8,fill=RoyalBlue] (k1) at (D3){};
 
   	\node[shape=circle,scale=1,draw] (A3) at (12,2.1){} ;
 	\node[shape=circle,scale=1,draw] (D3) at (13,3.7){} ;
	\node[shape=circle,scale=1,draw] (C3) at (14,2.1) {};
	\node[shape=circle,scale=1,draw] (B3) at (13,2.65){} ;

	\draw[thick] (A3) to (B3);		
	\draw[thick] (B3) to (C3);		
	\draw[thick] (B3) to (D3);		
	\draw[thick] (A3) to (12,1.75);	
	\draw[thick] (C3) to (14,1.75);	
	\draw[thick] (A3) to (11.7,2.3);	
	\draw[thick] (C3) to (14.3,2.3);		
	\draw[thick] (D3) to (13.3,3.9);	
	\draw[thick] (D3) to (12.7,3.9);	

 \node[shape=circle,scale=0.8,fill=red] (k1) at (C3){};
 \node[shape=circle,scale=0.8,fill=RoyalBlue] (k1) at (B3){};

 \node[shape=circle,scale=1,draw] (A3) at (5,-0.55){} ;
 	\node[shape=circle,scale=1,draw] (D3) at (6,1.05){} ;
	\node[shape=circle,scale=1,draw] (C3) at (7,-0.55) {};
	\node[shape=circle,scale=1,draw] (B3) at (6,0){} ;

	\draw[thick] (A3) to (B3);		
	\draw[thick] (B3) to (C3);		
	\draw[thick] (B3) to (D3);		
	
	\draw[thick] (A3) to (5,1.75-2.65);	
	\draw[thick] (C3) to (7,1.75-2.65);	
	\draw[thick] (A3) to (4.7,2.3-2.65);	
	\draw[thick] (C3) to (7.3,2.3-2.65);		
	\draw[thick] (D3) to (6.3,3.9-2.65);	
	\draw[thick] (D3) to (5.7,3.9-2.65);	

 \node[shape=circle,scale=0.8,fill=red] (k1) at (A3){};
 \node[shape=circle,scale=0.8,fill=RoyalBlue] (k1) at (C3){};
 
  \node[shape=circle,scale=1,draw] (A3) at (8.5,-0.55){} ;
 	\node[shape=circle,scale=1,draw] (D3) at (9.5,1.05){} ;
	\node[shape=circle,scale=1,draw] (C3) at (10.5,-0.55) {};
	\node[shape=circle,scale=1,draw] (B3) at (9.5,0){} ;

	\draw[thick] (A3) to (B3);		
	\draw[thick] (B3) to (C3);		
	\draw[thick] (B3) to (D3);		
	\draw[thick] (A3) to (8.5,1.75-2.65);	
	\draw[thick] (C3) to (10.5,1.75-2.65);	
	\draw[thick] (A3) to (8.2,2.3-2.65);	
	\draw[thick] (C3) to (10.8,2.3-2.65);		
	\draw[thick] (D3) to (9.8,3.9-2.65);	
	\draw[thick] (D3) to (9.2,3.9-2.65);

 \node[shape=circle,scale=0.8,fill=red] (k1) at (B3){};
 \node[shape=circle,scale=0.8,fill=RoyalBlue] (k1) at (C3){};

\node[scale=2.3] (arrowing) at (4.2,2.65){$\rightarrow$};
\node[scale=2.3] (arrowing) at (4.2,0){$\rightarrow$};

\node[scale=1.1] (arrowing) at (4,3.05){\hyperref[def:way1]{(a)}};
\node[scale=1.1] (arrowing) at (4,0.4){\hyperref[def:way2]{(b)}};

\node[scale=2.3] (arrowing) at (7.75,2.65){$\rightarrow$};
\node[scale=2.3] (arrowing) at (7.75,0){$\rightarrow$};

\node[scale=2.3] (arrowing) at (11.25,2.65){$\rightarrow$};


\end{tikzpicture}
\end{center}
\caption{\label{figureLemma} Transformations for $\eta$ to $\eta^{x,y}$ in $T^3$ where $x,y,z$ are neighbors of $w$. }
\end{figure}

In this section, we prove that the tagged particle $(X_t)_{t \geq 0}$ on $T^d$ satisfies a strong law of large numbers. As a first step, we show that $(X_t)_{t \geq 0}$ is \textbf{transient} for $d\geq 3$, i.e. $(X_t)_{t \geq 0}$ visits the root of $T^d$ almost surely only finitely many times. 
To do so, we use the framework introduced by Lyons et al. in order to study random walks on Galton-Watson trees, see \cite{lyons1995}. An infinite path $x_0,x_1,\dots$ of sites in $T^d$ will be denoted by $\overset{\rightarrow}{x}$. We say that a path $\overset{\rightarrow}{x}$ is a \textbf{ray} $\xi$ if it never backtracks, i.e. $x_i \neq x_{j}$ for all $i\neq j$. The set of rays starting at the root is called the \textbf{boundary} $\partial T^d$ of the tree $T^d$. We say that a path $\overset{\rightarrow}{x}$ \textbf{converges} to a ray $\xi$ if $\overset{\rightarrow}{x}$ visits every site at most finitely many times and $\xi$ is the unique ray which is intersected infinitely often. 
For a site $x$ and a ray $\xi$, let $[x,\xi]$ denote the unique ray starting in $x$ and converging to $\xi$. Moreover, 
for two distinct sites $x,y \in V$, let $x\wedge_{\xi} y$ denote the site where $[x,\xi]$ and $[y,\xi]$ meet for the first time. \\

For two vertices $x,y \in V$, recall that $|x-y|$ denotes the shortest path distance between $x$ and $y$. We define their \textbf{horodistance} with respect to some given ray $\xi$ as the signed distance
\begin{equation*}
\langle y-x \rangle_{\xi} := |y - x\wedge_{\xi} y| - |x - x\wedge_{\xi} y| \ .
\end{equation*} We set $\langle x\rangle_{\xi} := \langle x - o \rangle_{\xi}$ with respect to the root $o$ of $T^d$. 
Throughout the rest of this section, let $\xi \in \partial T^d$ be an arbitrary, but fixed boundary point of $T^d$, which will in the following be omitted as a subscript in the notation of the horodistance. 
Note that without loss of generality, we can define the addition on $T^d$ such that the horodistance defines a group homomorphism between $(T^d,+)$ and $(\mathbb{Z},+)$, i.e. 
\begin{equation}\label{eq:homomorphism}
\langle x+y\rangle= \langle x\rangle+\langle y\rangle
\end{equation} holds for all sites $x,y \in V$. \\

Our goal is to show a law of large numbers for the stochastic process $(\langle X_t\rangle)_{t \geq 0}$ from which we will deduce Theorem \ref{thm:LLN}. 
We define 
\begin{equation*}
\psi(\zeta) := \sum_{z \in V} p(|z|) (1-\zeta(z)) \langle z\rangle
\end{equation*} to be the \textbf{local drift at the root} for a configuration $\zeta \in \{ 0,1\}^{V}$ with $ \zeta(o) = 1$. Recall the definition of the environment process $(\zeta_t)_{t \geq 0}$ in \eqref{def:environment}. We want to express $(\langle X_t\rangle)_{t \geq 0}$ in terms of $(\zeta_t)_{t \geq 0}$. Observe that $(X_t, \zeta_t)_{t \geq 0}$ is a Feller process whose generator is given as the closure of
\begin{align*}
\tilde{\mathcal{L}}f(x,\zeta) &= \sum_{y,z \neq o}  p(|z-y|)\zeta(y)(1-\zeta(z))\left[ f(x,\zeta^{y,z})- f(x,\zeta)\right]  \\
&+ \sum_{y \in V}  p(|x-y|) (1-\zeta(y))\left[ f(y,\tau_{y-x}\zeta)- f(x,\zeta)\right]
\end{align*} 
 and let $(\mathcal{F}_t)_{t \geq 0}$ denote the respective $\sigma$-algebra. Note that the process $(X_t)_{t \geq 0}$ on its own is in general not Markovian. We now decompose the process $(\langle X_t \rangle )_{t \geq 0}$ into a martingale and a function depending only on the environment process. This follows the ideas of Proposition 4.1 in \cite[III]{liggett1999stochastic}.
 \begin{lemma}\label{lem:martingaleRelation} For all $t \geq 0$, it holds that
 \begin{equation}\label{eq:martingaleRelation}
 \langle X_t\rangle = \int_{0}^{t} \psi(\zeta_s) \dif s + M_t
 \end{equation} where $(M_t)_{t \geq 0}$ is a martingale with respect to $(\mathcal{F}_t)_{t \geq 0}$.
 \end{lemma}
 \begin{proof} Define the function $f(x,\zeta):= \langle x\rangle$. Observe that for this choice of $f$, we have that
 \begin{equation}\label{eq:generatorRelation}
 \tilde{\mathcal{L}}f(x,\zeta) = \sum_{y \in V} p(|x-y|) (1- \zeta(y))\left[ f(y,\tau_{y-x}\zeta)- f(x,\zeta)\right] = \psi(\zeta)
\end{equation} holds using \eqref{eq:homomorphism}. It remains to show that the process $(M_t)_{t \geq 0}$ defined via the relation in \eqref{eq:martingaleRelation} is indeed a martingale with respect to $(\mathcal{F}_t)_{t \geq 0}$. 
In particular, for all $s<t$, we need to verify that 
\begin{equation*}
\E\left[M_t-M_s | \mathcal{F}_s \right] =  0
\end{equation*} holds. 
Using the Markov property of $(X_t, \zeta_t)_{t \geq 0}$, we obtain that
\begin{align*}
\E\left[M_t-M_s | \mathcal{F}_s \right] 
&= \E\left[\langle X_t \rangle-\langle X_s \rangle - \int_{s}^t \psi(\zeta_r) \dif r | \mathcal{F}_s \right]  \\
&= \E^{(X_s,\zeta_s)}\left[\langle X_{t-s} \rangle-\langle X_0 \rangle - \int_{0}^{t-s} \psi(\zeta_r) \dif r \right] \ .
\end{align*} In particular, it suffices to show that for fixed $x\in V$ and $\zeta \in \{ 0,1\}^{V}$, we have that
\begin{equation*}
\E^{(x,\zeta)}\left[\langle X_{t} \rangle-\langle x \rangle\right] - \int_{0}^{t} \E^{(x,\zeta)}\left[\psi(\zeta_r)\right] \dif r = 0
\end{equation*}
holds for all $t \geq 0$. Using \eqref{eq:generatorRelation}, this follows immediately by Dynkin's formula.
 \end{proof} 
Applying the results of Section \ref{sec:environment}, we obtain the following lemma as an immediate consequence and an analogue of Corollaries 4.5 and 4.16 in \cite[III]{liggett1999stochastic}.
 \begin{lemma} \label{lem:martingaleErgodicity} Suppose that $(\eta_t)_{t \geq 0}$ has initial distribution $\nu_{\rho}^{\ast}$ for some $\rho \in [0,1]$. Then the martingale $(M_t)_{t \geq 0}$ in Lemma \ref{lem:martingaleRelation} has stationary and ergodic increments.
 \end{lemma}
 \begin{proof} Observe that $\langle X_t \rangle$ can be expressed as a function $F_t$ of $\{\zeta_s, 0 \leq s \leq t\}$ for all $t\geq 0$ since all transitions of $(\langle X_t \rangle)_{t \geq 0}$ correspond precisely to the shifts in the environment process. In particular, we have that
 \begin{equation*}
\langle X_t \rangle - \langle X_0 \rangle = F_t(\zeta_s, 0 \leq s \leq t) 
\end{equation*} holds. Using that $(\zeta_t)_{t \geq 0}$ is stationary, we have by Lemma \ref{lem:martingaleRelation} that
 \begin{equation*}
 M_t - M_s = F_{t-s}(\zeta_r, s \leq r \leq t) + \int_{s}^{t} \psi(\zeta_s) \dif s
 \end{equation*}
holds for all $s<t$. Recall from Propositions \ref{pro:invarianceEnvironment} and \ref{pro:ergodicity} that the Palm measure is stationary and ergodic for $(\zeta_t)_{t \geq 0}$. Hence, the claimed statement follows. 
 \end{proof}
Next, we prove a law of large numbers for the process $(\langle X_t \rangle)_{t\geq 0}$.
An analogous statement for the tagged particle on $\mathbb{Z}^{d}$ can be found as Theorem 4.17 in \cite[III]{liggett1999stochastic}.
\begin{proposition}\label{pro:LLNhorodistance} For $d\geq 2$, let $(\eta_t)_{t\geq 0}$ on $T^d$ have initial distribution $\nu_{\rho}^{\ast}$ for some $\rho \in (0,1)$. Then the associated tagged particle $(X_t)_{t\geq 0}$ satisfies
\begin{equation}\label{eq:expectationHorodistance}
\E\left[ \langle X_t\rangle  \right] = (1-\rho)(d-2)\sum_{i \in \N_0}i p(i)
\cdot t = v\cdot t
\end{equation} for all $t\geq 0$. Moreover, we have that $\P_{\nu_{\rho}^{\ast}}$-a.s.
\begin{equation}\label{eq:LLNhorodistance}
\lim_{t \rightarrow \infty} \frac{\langle X_t\rangle }{t} =  v
\end{equation} holds. In particular, the tagged particle $(X_t)_{t \geq 0}$ on $T^d$ is transient for $d\geq 3$.
\end{proposition}
 \begin{proof} For $d=2$, the statement follows by symmetry, so we assume that $d \geq 3$. Taking expectations on both sides of \eqref{eq:martingaleRelation} in Lemma \ref{lem:martingaleRelation} yields that
 \begin{equation*}
 \E\left[ \langle X_t\rangle \right] = \int_{0}^{t} \E\left[\psi(\zeta_s)\right] \dif s + \E[M_t] \ .
\end{equation*} Observe that $(\psi(\zeta_t))_{t \geq 0}$ is a stationary sequence. Hence, we have that
\begin{equation*}
\E\left[\psi(\zeta_t)\right] = \E\left[\psi(\zeta_0)\right] = (1-\rho)(d-2)\sum_{i \in \N_0}i p(i)
\end{equation*} holds for all $t\geq 0$.
 Since $(M_t)_{t \geq 0}$ is a martingale, the statement in \eqref{eq:expectationHorodistance} follows. In order to show \eqref{eq:LLNhorodistance}, recall Proposition \ref{pro:ergodicity} and Lemma \ref{lem:martingaleErgodicity} and apply the ergodic theorem to both terms on the right-hand side of \eqref{eq:martingaleRelation}, respectively. 
\end{proof}
%
As an immediate consequence of the transience of the tagged particle $(X_t)_{t \geq 0}$ for $d \geq 3$ and $T^d$ being spherically symmetric, we obtain the following corollary.
\begin{corollary}\label{cor:LLN} For $d\geq 3$ and $\rho \in (0,1)$, let $\overset{\rightarrow}{x}$ denote the trajectory of the tagged particle $(X_t)_{t \geq 0}$ on $T^d$. Then $\overset{\rightarrow}{x}$ converges almost surely to a unique boundary point $x_{+\infty} \in \partial T^d$. Moreover, for any deterministic choice of $\xi \in \partial T^d$, we have that $x_{+\infty} \neq \xi$ holds almost surely.
\end{corollary}
\begin{proof}[Proof of Theorem \ref{thm:LLN}]
By Corollary \ref{cor:LLN}, we almost surely have for all $t\geq 0$ sufficiently large that
\begin{equation}\label{eq:estimateHorodistance}
|X_t| = \langle X_t \rangle + 2 |w| 
\end{equation} holds where $w$ is the last common vertex of $x_{+\infty}$ and $\xi$. Since $\xi$ was arbitrary, but fixed at the beginning, we have that $w$ is well defined and $|w|$ is almost surely finite. Since $|w|$ does not depend on $t$, we obtain Theorem \ref{thm:LLN} from Proposition \ref{pro:LLNhorodistance}. 
\end{proof}

\section{Diffusivity of the tagged particle}\label{sec:diffusivity}

In order to prove Theorem \ref{thm:CLT}, we show a central limit theorem for the process $(\langle X_t\rangle)_{t \geq 0}$. Recall from \eqref{eq:martingaleRelation} that $(\langle X_t \rangle)_{t \geq 0}$ can be decomposed into a martingale $(M_t)_{t \geq 0}$ and a process $\int_{0}^{t} \psi(\zeta_s) \dif s$. For $p(.)$ and $v$ taken from Section \ref{sec:main}, we define
\begin{equation*}
\bar{\psi}(\zeta) := \psi(\zeta) - v =  \sum_{x\in V}p(|x|)\langle x \rangle (\rho- \zeta(x)) \ . 
\end{equation*} Our goal is to establish a similar decomposition for the process $\int_{0}^{t} \bar{\psi}(\zeta_s) \dif s$.  Let $L^2(\nu_{\rho}^{\ast})$ denote the Hilbert space of square integrable functions with respect to $\nu_{\rho}^{\ast}$ and scalar product
\begin{equation*}
\langle f,g \rangle_{\nu_{\rho}^{\ast}} := \int f g \dif \nu_{\rho}^{\ast} \ .
\end{equation*} Observe that the environment process $(\zeta_t)_{t \geq 0}$ with generator $L$ of \eqref{def:generatorEnvironment} is reversible with respect to $\nu_{\rho}^{\ast}$ for all $\rho \in [0,1]$. For a function $f \in L^2(\nu_{\rho}^{\ast})$ in the domain of $L$, we define its $\boldsymbol{\lVert . \rVert_{1}}$\textbf{-norm} to be
\begin{equation*}
\lVert f \rVert_{1} :=  \sqrt{\langle f,(-L)f \rangle_{\nu_{\rho}^{\ast}}} \ .
\end{equation*} Let $\mathcal{H}_1$ denote the respective Hilbert space generated by all local functions $f$ of finite $\lVert . \rVert_{1}$-norm. We define its dual space $\mathcal{H}_{-1}$ to be the Hilbert space generated by all local functions which have a finite norm with respect to
\begin{equation*}
\lVert f \rVert_{-1} := \inf\left\lbrace C \geq 0 \colon \abs{ \int f g  \dif \nu_{\rho}^{\ast} } \leq C \lVert g \rVert_{1} \text{ for all local functions }g \right\rbrace \ .
\end{equation*}
The following result was shown by Sethuraman et al. for the exclusion process on $\mathbb{Z}^d$ with $d \geq 3$ and carries over to $T^d$ for $d\geq 3$, see  \cite[Lemma 2.1]{sethuraman2000}.
\begin{proposition}\label{pro:sethuraman} For $d \geq 3$, we have that $\bar{\psi} \in \mathcal{H}_{-1}$ holds.
\end{proposition}
Note that their proof only uses transience of the simple random walk on the underlying graph $T^d$ as well as the fact that $\bar{\psi}$ is a bounded, local function of zero mean.
The next proposition is a special case of the celebrated theorem by Kipnis and Varadhan on additive functional of reversible Markov processes, see \cite[Theorem 1.8]{kipnis1986central}.
\begin{proposition} \label{pro:kipnis} Assume that $\bar{\psi} \in L^2(\nu_{\rho}^{\ast}) \cap \mathcal{H}_{-1}$ has mean zero. Then $\int_{0}^{t} \bar{\psi}(\zeta_s) \dif s$ can be decomposed into a square integrable martingale $(N_t)_{t \geq 0}$ with stationary increments and a stochastic process $(R_t)_{t \geq 0}$, i.e.
\begin{equation*}
\int_{0}^{t} \bar{\psi}(\zeta_s) \dif s = N_t + R_t 
\end{equation*} where $(R_t)_{t \geq 0}$ satisfies $\lim\limits_{t \rightarrow \infty}t^{-1}\cdot\E\left[ R_t^2 \right] = 0$.
\end{proposition} 
\begin{proof}[Proof of Theorem \ref{thm:CLT}]A simple computation shows that the martingale $(M_t)_{t \geq 0}$ satisfies a CLT, see \cite[Proposition 4.19]{liggett1999stochastic}. Combining Propositions \ref{pro:sethuraman} and \ref{pro:kipnis}, we can now apply a martingale central limit theorem to the process $(M_t+N_t)_{t \geq 0}$. Note that the limit variance is non-degenerate using the arguments of Section 6.9 in \cite{komorowski2012}. Together with \eqref{eq:estimateHorodistance}, this yields Theorem \ref{thm:CLT}.
\end{proof}

\textbf{Acknowledgments } The first two authors were supported by the NSFC grant No. 11531001. The second author now works at the Central  University of Finance and Economy, Beijing, China. The last author thanks the Studienstiftung des deutschen Volkes for financial support.

\bibliographystyle{plain}
\bibliography{SpeedRegular}

\end{document}